\documentclass{birkjour}
 \newtheorem{thm}{Theorem}[section]
 \newtheorem{cor}[thm]{Corollary}
 \newtheorem{lem}[thm]{Lemma}
 
 \theoremstyle{definition}
 \newtheorem{defn}[thm]{Definition}
 \theoremstyle{remark}

 \numberwithin{equation}{section}

\begin{document}

%-------------------------------------------------------------------------
% editorial commands: to be inserted by the editorial office
%
%\firstpage{1} \volume{228} \Copyrightyear{2004} \DOI{003-0001}
%\seriesextraline{This is the Concrete Title of this Book\br H.E. R and S.T.C. W, Eds.}
%\firstpage{1}

\title[On tensor analogues of some automorphisms]
 {On  tensor analogues of  commuting automorphisms and  central automorphisms}

%----------Author 1

\author[ Mohammadzadeh]{Fahimeh Mohammadzadeh}
\address{Department of Mathematics\\
 Faculty of Sciences\\
 Payame Noor University\\
 19395-4697 Tehran\\
 Iran}
\email{F.mohamadzade@gmail.com}

%\thanks{}
%----------Author 2

\author[ Golmakani]{Hanieh Golmakani}
\address{Department Of Mathematics\\
 Faculty Of Science \\
 Mashhad Branch \\
 Islamic Azad University \\
 91735-413 Mashhad\\
 Iran}
\email{h.golmakani@mshdiau.ac.ir}

%---------------------------

\author[Hokmabadi]{Azam Hokmabadi}

\address{%
Department of Mathematics\\
 Faculty of Sciences\\
 Payame Noor University\\
 19395-4697 Tehran\\
 Iran}

\email{ahokmabadi@pnu.ac.ir}

%-------------------------------

\author[Mohammadzadeh]{Elaheh Mohammadzadeh}
\address{Department of Mathematics\\
 Faculty of Sciences\\
 Payame Noor University\\
 19395-4697 Tehran\\
 Iran}
\email{mohamadzade.pnus@gmail.com}



\date{January 1, 2004}

\begin{abstract}

In this paper, we introduce the tensor analogues of commuting automorphisms and central automorphisms. Then we give several properties of such automorphisms and apply these new concepts to give some interesting results for $2_\otimes$-Engel elements.

\end{abstract}

%%% ----------------------------------------------------------------------
\maketitle
%%% ----------------------------------------------------------------------
%\tableofcontents
\section{Introduction and Motivation}

For a group $G$ the non-abelian tensor square $G \otimes G$ is a
group generated by the symbols $g \otimes h$, subject to the
relations
\begin{center}
$gg^{'}\otimes h =(g^{g^{'}}\otimes h^{g^{'}})(g^{'} \otimes h )$
and  $g \otimes hh^{'} =    (g \otimes h^{'} ) (g^{h^{'}}\otimes
h^{h^{'}})$,
\end{center}
where $g,g^{'},h,h^{'} \in G $ and $g^{h}=h^{-1}gh$.\\
The more general concept of non-abelian tensor product of groups
acting on each other in certain compatible way was introduced by
Brown and etc. in \cite{Bro}, following the ideas of Dennis
\cite{Den}. Since the concept of tensor product holds a number of applications in many other branches of group theory as well as certain other fields of mathematics, it  had been one of the concerns to many recent studies. For further research you may consider \cite{Nir},\cite{Berg}, \cite{Ben}.\\
The following lemmas are very important in non-abelian tensor square and frequently used in this paper.\\

\begin{lem} \cite{Brown} Let $g, g^{'}, h, h^{'} \in G$.
The following relations
hold in $G \otimes G$\\
(i) $(g^{-1}\otimes h ) ^{g}=(g \otimes h )
^{-1}=(g\otimes h ^{-1}) ^{h}.$\\
(ii) $(g^{'}\otimes h ^{'} ) ^{g \otimes
h}=(g^{'}\otimes h ^{'} ) ^{[g, h ]}$.\\
(iii) $g^{'} \otimes [g, h ]=(g\otimes h )
^{-g^{'}}(g\otimes h ).$\\
(iv) $[g, h ] \otimes g^{'}=(g\otimes h )^{-1}(g\otimes h
)^{g^{'}}$.\\
(v) $[g, h ] \otimes [g ^{'}, h ^{'} ]=[ g \otimes g ^{'}, h
\otimes  h ^{'} ]$.\\
Note that $G$ acts on $G \otimes G$ by $(g\otimes h )^{g^{'}}=
g^{g^{'}} \otimes h ^{g^{'}}$.
\end{lem}

\begin{lem} \cite{Brown} Let $G$ be a group.

(i) There exists a unique isomorphism $\theta : G \otimes G \rightarrow G \otimes G$ such that

\hspace{0.4 cm} $\theta (g\otimes h)=(h\otimes g)^{-1}$.

(ii) For any homomorphism $\alpha : G \rightarrow G$ there is a unique homomorphism

\hspace{0.6 cm}$\alpha \otimes \alpha: G \otimes G \rightarrow G \otimes G$ such that
$ \alpha \otimes \alpha (g \otimes h)= \alpha(g) \otimes \alpha(h)$, for

\hspace{0.5 cm} all $g,h \in G$.
\end{lem}

There is a natural surjective homomorphism $\kappa : G \otimes G \rightarrow G'$ that is defined by $\kappa (g \otimes h) = [g, h]$. This homomorphism suggests generalizations of many commutator-related concepts to tensor analogues. In 1995, Ellis \cite{Ell} defined the tensor center of $G$,  $Z^\otimes(G)$, consisting of all $a \in G$ with $a\otimes x=1_\otimes$, for every $x \in G$.
In 2003, Biddle and Kappe \cite{Bid} introduced the $n^{\text{\tiny th}}$ tensor
center of $G$,$$ Z_n^\otimes (G)=\{ a \in G | [a,g_1, g_2, \dots ,
g_{n-1}]\otimes g_n=1_\otimes,\ \ \forall \ g_1, \dots , g_n \in G \}.$$
It can be easily seen that $Z_n^\otimes (G)$ is a
characteristic subgroup of $G$ contained in $Z_n(G)$.
Further, for a group $G$ and a non-empty subset $X$, they \cite{Bid} defined the
tensor annihilator of $X$ in $G$ to be the set
$$C_G^\otimes(X)=\{ a \in G | a\otimes x=1_\otimes,\ \ \forall \ x \in X \}.$$
 Moreover the tensor analogue of right 2-Engel elements, the right $2_\otimes$-Engel elements, were introduced by Biddle and Kappe \cite{Bid} as follows:
$$R_2^\otimes(G)= \{g \in G | [g,x]\otimes x=1_\otimes, \forall \ x \in G \}.$$
They proved that $R_2^\otimes(G)$ is a characteristic subgroup of $G$ that contains $Z(G)$ and is contained in
$R_2(G)$. They also gave an example of a group $W$ in which $R_2^\otimes(W)$ is properly contained in the set of right 2-Engel elements and the center is properly contained in $R_2^\otimes(W)$. It is interesting to know that when $Z(G)$ and $R_2(G)$ coincide exactly with $R_2^\otimes(G)$. The present study gives a condition for $G$ which implies $Z(G)=R_2^\otimes(G)$.

In 2005, P. Moravec \cite{Mor} determined some further information about $R_2^\otimes(G)$ and provide some new characterizations of this subgroup. He also described the structure of $2_\otimes$-Engel groups. Recall that a group $G$ is called $2_\otimes$-Engel group, if $R_2^\otimes(G)=G$. Among the results of the present research, we give an interesting result for the derived subgroup of $R_2^\otimes(G)$ \\

In this paper, we first introduce the tensor analogues of commuting automorphisms and central automorphisms. Then we apply these concepts
 to give some interesting results for $R_2^\otimes(G)$. We recall that an automorphism $\alpha$ of a group $G$ is called a
commuting automorphism, if $g \alpha(g)=\alpha(g)g$, for all $g \in G$. The set of all commuting automorphisms of $G$  is denoted by $A(G)$. This kind of automorphisms were first studied for various classes of rings \cite{Bell,Div,Luh}.
M. Desconesco and G. Silberbers \cite{DS} proved that $A(G)$ do not necessarily form a subgroup of $Aut(G)$.
We also recall that a central automorphism of a group $G$ is an automorphism $\alpha$ of $G$ such that for all $x \in G$, $[g, \alpha] \in Z(G)$, in which $[g,\alpha]=g^{-1}\alpha(g)$. The set of all central automorphisms of a group $G$ is usually denoted by $Aut_c(G)$.
 It is clear that $Aut_c(G)$ is a subgroup of $Aut(G)$ which is contained in $A(G)$.\\

In Section 2,  we first introduce the tensor analogues of commuting automorphisms of $G$, called tensor commuting automorphisms and
denote the set of all such automorphisms by $A^\otimes(G)$. Then we give several properties of elements of $A^\otimes(G)$.
Moreover, we give an interesting result which bridges together the right $2_\otimes$-Engel elements and tensor commuting automorphisms and then apply it to find a condition for $G$ which implies that $Z(G)= R_2^\otimes(G)$.
In the third section, we first introduce the tensor central automorphisms of $G$ and
denote the set of all such automorphisms by $Aut_c^\otimes(G)$. Then we give some properties of $Aut_c^\otimes (G)$ and apply them to obtain some interesting results for  $R_2^\otimes(G)$ and $A^\otimes(G)$.

\section{Tensor commuting automorphisms}

\begin{defn}
 We call an automorphism $\alpha$ of a group
$G$ a tensor commuting automorphism, if $g \otimes
\alpha(g)=1_\otimes,$ for all $g \in G$. We denote the set of all tensor
commuting automorphisms of $G$ by $A^\otimes(G)$.\
\end{defn}

The following lemma which is frequently used in this paper, gives an important property of a tensor commuting automorphism.\\

\begin{lem} Let $G$ be a group and $\alpha \in
A^\otimes(G)$. Then for all $x, y \in G$,\\
(i) $x \otimes \alpha(y)=(y \otimes \alpha(x))^{-1}.$\\
(ii) $\alpha(x) \otimes y=(\alpha(y) \otimes x)^{-1}.$
\end{lem}

\begin{proof}
Let $x, y \in G$ and $\alpha \in
A^\otimes(G)$. Then we have
 \begin{eqnarray*}
1_\otimes &=& xy^{-1} \otimes \alpha(xy^{-1})
\\&=& xy^{-1} \otimes \alpha(x)\alpha(y^{-1})\\&=&
(x \otimes \alpha(y)^{-1})^{y^{-1}}(y^{-1} \otimes \alpha(x))^{\alpha(y)^{-1}}.\\
\end{eqnarray*}
The above equalities imply that $(x \otimes \alpha(y)^{-1})^{-y^{-1}}=(y^{-1} \otimes
\alpha(x))^{\alpha(y)^{-1}}$. Then by Lemma 1.1, we have $(x \otimes
\alpha(y))^{\alpha(y)^{-1}y^{-1}}=(y \otimes
\alpha(x))^{-y^{-1}\alpha(y)^{-1}}.$
Thus we have $(y \otimes \alpha(x))^{-1}=(x \otimes
\alpha(y))^{[\alpha(y),y]}$. On the other hand, $\alpha \in A^{\otimes}(G)$ implies that $[\alpha(y),y]=1$.
Therefore (i) holds. The proof of (ii) is similar to (i).

\end{proof}

The next theorem shows that $A^{\otimes}(G)$ does have many properties of a group.\\

\begin{thm} Let $G$ be a group.\\
(i) If $x\otimes x=1_\otimes$, for all $x \in G$, then  $A^\otimes(G)$ is closed under powers.\\
(ii) $A^\otimes(G)$ is closed under conjugates.\\
(iii) $A^\otimes(G)$ is closed under inverses.\\
(iv) If $\alpha, \beta \in A^\otimes(G)$, then $\alpha \beta \in
A^\otimes(G)$, if and only if $\alpha(x) \otimes
\beta(x)=1_\otimes$, for all $x \in G$.\\
(v) If  $\alpha, \beta \in A^\otimes(G)$, then $(\alpha \beta)^n
\alpha, \beta (\alpha  \beta)^n \in A^\otimes(G)$, for all $n \geq 0$.
\end{thm}

\begin{proof} (i)  Let  $\alpha \in A^\otimes(G)$ and $n \geq 0$. We use induction on $n$ to show that $\alpha ^ n \in
A^\otimes(G)$. The assertion is clear for $n=0,1$. If $n=2$, then by Lemma 2.2 (i),
$x \otimes \alpha^2(x)= (\alpha(x) \otimes \alpha(x))^{-1}=1_\otimes$ and so $\alpha^2 \in A^\otimes(G)$.
Suppose that $n > 2$ and $\alpha ^ k \in A^\otimes(G)$, for all $k<n$. Then $\alpha^{n-2} \in A^\otimes(G)$ and
we have $y \otimes \alpha^{n-2}(y)=1_{\otimes}$, for all $y \in G$. Hence by Lemma 2.2 (i), we have
\begin{eqnarray*}
x \otimes \alpha ^n(x) &=& (\alpha^{n-1} (x)\otimes \alpha (x))^{-1}\\
& =& (\alpha^{n-2}( \alpha(x) )\otimes \alpha (x))^{-1}=1_{\otimes}.
\end{eqnarray*}

Therefore $\alpha ^ n \in A^\otimes(G)$ and the assertion follows.\\

(ii) Let  $\alpha \in A^\otimes(G) , \gamma \in Aut(G)$ and $x \in G$. Then
$\gamma(x) \otimes \alpha(\gamma(x))=1_\otimes$.
 Hence in view of Lemma 1.2 (ii), we have $x\otimes \gamma^{-1} \alpha \gamma(x) =\gamma
^{-1} \otimes \gamma^{-1} ( \gamma(x) \otimes \alpha
\gamma(x))=1_\otimes$. Therefore $\gamma^{-1} \alpha \gamma \in A^\otimes(G) $.\\

(iii) The result follows by Lemma 1.2.\\

(iv) Since by Lemma 2.2 (i) $\beta(x) \otimes \alpha(x)=(x \otimes \alpha \beta(x))^{-1}$, for all $x \in
G$, we have $\alpha \beta \in A^\otimes(G)$ if and only if  $\beta(x) \otimes \alpha(x)=1_\otimes $. Hence the result holds by  Lemma 1.2 (i).

(v) We use induction on $n$.
Since $\alpha , \beta \in A^\otimes(G)$, we have $\beta(x)\otimes \alpha(\beta(x))=1_\otimes$.
Then applying Lemma 2.2 (i) and Lemma 1.2 (i), we have
$$x \otimes \beta \alpha \beta(x)=(\alpha(\beta(x)) \otimes \beta(x))^{-1}=1_\otimes,$$ for all $x \in G$.
 Hence  $\beta \alpha \beta \in A^\otimes(G)$ and the result holds for $n=1$. Now assume that the assertion holds for $n>1$. Then $\beta(\alpha
\beta)^n, \alpha(\beta \alpha)^n \in A^\otimes(G),$ for all $\alpha ,\beta  \in A^\otimes(G)$. Hence we have
 $\beta(\alpha\beta)^{n+1}=\beta(\alpha(\beta \alpha)^n)\beta$ that belongs to $A^\otimes(G)$,
  by the first step of induction. Similarly, we have  $\alpha(\beta \alpha)^{n+1}\in A^\otimes(G)$ and hence the assertion follows.
\end{proof}

\begin{lem} Let $G$ be a group and $x \in G$. If
$\alpha \in A^\otimes(G)$, then $[x,\alpha] \in C_G(G \otimes G)$.
\end{lem}

\begin{proof} Let $x,y,z$ be arbitrary elements of  $G$ and $\alpha \in A^\otimes(G)$. By Lemma 2.2, we have
$(yx^{-1} \otimes \alpha(z))^{-1}=z \otimes \alpha(yx^{-1}).$
Then we have $$(x^{-1} \otimes \alpha(z))^{-1}(y \otimes \alpha(z))^{-x^{-1}}=(z \otimes \alpha(x^{-1}))(z \otimes \alpha(y))^{\alpha (x^{-1})}.$$
Now applying Lemma 2.2, we conclude that $(z \otimes \alpha(y))^{\alpha (x^{-1})}=(y \otimes \alpha(z))^{-x^{-1}}=(z \otimes \alpha(y))^{x^{-1}}.$
Then we have $(z \otimes \alpha(y))^{[x, \alpha]}=z \otimes \alpha(y)$, and therefore $[x,\alpha] \in C_G(G \otimes G)$.
\end{proof}

The next interesting result bridges together the right $2_\otimes$-Engle elements and tensor
commuting automorphisms of a group.\\

\begin{thm}
Let $G$ be a group and $x\otimes x=1_\otimes$, for all $x \in G$. Assume that $T_g$ is the inner
automorphism of $G$ induced by element $g $ of  $G$. Then $T_g \in A^\otimes (G)$ if and only if $g \in R_2^\otimes (G)$.
\end{thm}

\textbf{\bf Proof}. By Definition 2.1,  $T_g \in A^\otimes (G)$ if and only if $x^g \otimes x= T_g(x)\otimes x=1_\otimes$, for all
$x \in G$. Also, by Lemma 1.1 (i), for all $x \in G$, we have
\begin{eqnarray*}
x^g\otimes x=1_\otimes &\Longleftrightarrow& [g,x^{-1}]x\otimes x=1_\otimes
\\&\Longleftrightarrow&([g,x^{-1}]\otimes x)^x=1_\otimes
\\&\Longleftrightarrow& [g,x]^{-x^{-1}}\otimes x=1_\otimes
\\&\Longleftrightarrow& [g,x]\otimes x =1_\otimes.\\
\end{eqnarray*}
Finally, the last identity holds for all $x \in G$ if and only if $g \in R_2^\otimes (G)$.\\

An interesting consequence of Theorem 2.5, is the following corollary. \\

\begin{cor} Let $G$ be a group and $x\otimes x=1_\otimes$, for all $x \in G$. Then $A^\otimes (G) \cap Inn(G)$
is a subgroup of $Aut(G)$ and $$A^\otimes (G) \cap Inn(G) \cong
\frac{R_2^\otimes (G)}{Z(G)}.$$
\end{cor}

\begin{proof}

Consider the map $\theta: R_2^\otimes
(G)\rightarrow   Inn(G)$, defined by
$\theta(g)=T_g$, where $T_g$ is the inner automorphism of $G$
induced by $g$. Then it is easy to see that $\theta$ is an
homomorphism with $\ker \theta =Z(G)$. Also, by Theorem 2.5 $\theta (R_2^\otimes (G))=A^\otimes (G) \cap Inn(G)$. Therefore $A^\otimes (G) \cap Inn(G)$ is a subgroup of $Aut(G)$ and $A^\otimes (G) \cap Inn(G) \cong \frac{R_2^\otimes (G)}{Z(G)}.$
\end{proof}

The following important result is an immediate consequence of Corollary 2.6.\\

\begin{cor} Let $G$ be a group and $x\otimes
x=1_\otimes$, for all $x \in G$. If $A^\otimes (G)$ is the identity subgroup of $Aut(G)$,
then $ R_2^\otimes (G)=Z(G)$.
\end{cor}

\section{Tensor central automorphisms}

In this section, we define the new concept of tensor central automorphism. The set of all tensor central automorphisms which is denoted by $Aut_c^\otimes(G)$, forms a group. Then we give some properties of this subgroup of $Aut(G)$ and apply them to obtain some interesting results for $R^\otimes(G)$ and $A^\otimes(G)$.\\

\begin{defn} We call an automorphism $\alpha$ of
a group $G$, a tensor central automorphism, if $[x, \alpha] \in
Z^\otimes (G)$, for all $x \in G$.
It is straightforward to see that the set of all
tensor central automorphisms of $G$, denoted by $Aut_c^\otimes(G)$, is a subgroup of $Aut(G)$.
\end{defn}

\begin{lem}
Let $G$ be a group and $x\otimes x=1_\otimes$, for all $x \in G$. Then $Aut_c^\otimes(G) \subseteq A^\otimes (G)$.
\end{lem}

\begin{proof} Let $\alpha \in Aut_c^\otimes(G)$ and $x \in G$. Then we have $x^{-1}\alpha(x) \in Z^\otimes(G)$ and so
 $x^{-1}\alpha(x)\otimes x=1_\otimes$. This implies that $\alpha(x)\otimes x=1_\otimes$. Hence $\alpha \in A^\otimes(G)$, by Lemma 1.2 (i).
 \end{proof}

The next lemma provides some criterions for an automorphism to be tensor central.\\

\begin{lem}
Let $G$ be a group and $\alpha \in  Aut(G)$. Then \\
(i) $\alpha \in Aut_c^\otimes(G)$ if and only if $\alpha(y)\otimes x=y \otimes x$, for all $x,y \in G$.\\
(ii) $\alpha \in Aut_c^\otimes(G)$ if and only if $(x \otimes \alpha(y))= (x\otimes y)^{[y,\alpha]}$, for all $x,y \in G$.\\
(iii) If $\alpha \in Aut_c^\otimes(G)$, for all $x,y \in G$, then $\alpha  \otimes \alpha=id_{G \otimes G}.$\\
(iv) Let $ \alpha \in A^\otimes(G)$. Then $\alpha^2 \in Aut_c^\otimes(G)$ if and only if  for all $x,y \in G$, $(\alpha \otimes \alpha)(x \otimes y)=(y \otimes x)^{-1}$.
\end{lem}

\begin{proof} (i) Let $x,y \in G$.
Then we have
$$y \otimes x=y \alpha(y^{-1})\alpha(y) \otimes x= ([ y, \alpha] \otimes x)^{\alpha(y)}(\alpha(y) \otimes x).$$
The above equalities show that $\alpha (y) \otimes x=y\otimes x$, for all $x,y \in G$,
if and only if $[y,\alpha] \in Z^ \otimes(G)$, for all $y \in G$, and this
holds if and only if $\alpha \in Aut_c^\otimes(G)$.\\

(ii) The proof is similar to (i), but we use the following equalities
$$x\otimes y=x\otimes \alpha(y)\alpha(y^{-1})y=(x \otimes [ y, \alpha]^{-1})(x\otimes \alpha(y))^{[ y, \alpha]^{-1}}.$$

(iii) Using part (i) and the assumption, for all $x,y \in G$ we have
$$x\otimes y=\alpha(x)\otimes y=\alpha^2(x)\otimes y=\alpha(x)\otimes \alpha(y)=\alpha \otimes \alpha(x \otimes y).$$

(iv) We have
\begin{eqnarray*}
\alpha^2 \in Aut_c^\otimes(G) &\Longleftrightarrow& [y,\alpha^{2}] \in Z^{\otimes}(G), \forall y \in G
\\&\Longleftrightarrow& y^{-1} \alpha^2(y) \otimes x=1_\otimes,   \forall x \in G
\\&\Longleftrightarrow& ( y^{-1}  \otimes x)^{\alpha^2(y)} (  \alpha^2(y) \otimes x)=1_\otimes
\\&\Longleftrightarrow& \alpha^2(y) \otimes x= ( y^{-1}  \otimes x)^{ - \alpha^2(y)} .\\
\end{eqnarray*}

Also $ ( y^{-1}  \otimes x)^{ - \alpha^2(y)}= ( y \otimes x)^{ y^{-1} \alpha^2(y)}=y \otimes x$ and $ \alpha^2(y) \otimes x= (\alpha(x) \otimes \alpha (y))^{-1}= (\alpha \otimes \alpha (x \otimes y ))^{-1}$. Therefore $y \otimes x=( \alpha \otimes \alpha (x \otimes y ))^{-1}$. This complete the proof.
\end{proof}

It is interesting to know that when $A^\otimes(G)$ is a subgroup of $Aut(G)$.
The next result applies $Aut_c^\otimes(G)$ to give a sufficient condition for $A^\otimes(G)$ to
be a subgroup of $Aut(G)$.\\

\begin{thm} Let $G$ be a group  such that $x\otimes
x=1_\otimes$ for all $x \in G$. Assume that $(A^\otimes(G))'\subseteq
Aut^\otimes_c(G)$ and for all $\alpha \in Aut(G)$, the identity
$\alpha(x)\otimes x=x\otimes \alpha(x)$ implies that $x\otimes
\alpha(x)=1_\otimes$. Then $A^\otimes(G)$ is a subgroup of
$Aut(G)$.
\end{thm}

\begin{proof} Since $A^\otimes(G)$ is closed under inverses by Theorem 2.3 (iii), it is enough to show that $A^\otimes(G)$ is closed under multiplication.
Let $\alpha,\beta \in A^\otimes(G)$. By
hypothesis $[\alpha,\beta] \in Aut_c^\otimes(G)$ and so
$[x,[\alpha,\beta]] \in Z^\otimes (G)$, for all $x \in G$. Then we have $x^{-1}[\alpha, \beta](x) \otimes \alpha \beta(x)=1$ and this implies that
$$[\alpha, \beta](x) \otimes \alpha \beta(x)=(x \otimes \alpha \beta(x))^{[x, [\alpha, \beta]]}= x \otimes \alpha \beta(x).$$
Then applying Lemma 2.2, we have
\begin{eqnarray*}
 x \otimes \alpha \beta(x) &=&
( \beta(x)\otimes \alpha([\alpha, \beta](x)))^{-1}\\
&=&  ( \beta(x)\otimes \beta^{-1} \alpha \beta(x))^{-1}\\
&=& \alpha \beta(x) \otimes x.
\end{eqnarray*}
Then by the assumption of theorem, we have $x \otimes \alpha
\beta (x)=1_\otimes$ and therefore $ \alpha \beta \in A^\otimes (G)$ and hence the assertion follows.
\end{proof}

The next interesting result bridges together the tensor central automorphisms of a group $G$ and the elements of $Z_2^\otimes(G)$.\\

\begin{thm} Let $G$ be a group and $T_g$ be the inner automorphism of $G$ induced by element $g$ of
$G$. Then $T_g \in Aut_c^\otimes(G)\cap Inn(G)$ if and only if $g
\in Z_2^\otimes(G)$.
\end{thm}

\begin{proof} By Lemma 3.3 (i), we have

\begin{eqnarray*}
T_g \in Aut_c^\otimes(G)\cap Inn(G)&\Longleftrightarrow&  T_g(y)\otimes x= y\otimes x, \hspace{0.5cm} \forall x,y \in G\\
&\Longleftrightarrow&  y^g \otimes x= y\otimes x, \hspace{1.0 cm} \forall x,y \in G\\
&\Longleftrightarrow&  [g,y^{-1}]y \otimes x= y\otimes x, \hspace{0.08cm} \forall x,y \in G\\
&\Longleftrightarrow& ([g,y^{-1}] \otimes x)=1_\otimes, \hspace{0.4cm} \forall x,y \in G\\
&\Longleftrightarrow& [g,y^{-1}] \in Z^\otimes(G), \hspace{0.75cm} \forall y \in G\\
&\Longleftrightarrow& g \in Z_2^\otimes(G).\\
\end{eqnarray*}

\end{proof}

\begin{cor} Let $G$ be a group. Then
$Aut_c^\otimes(G)\cap Inn(G)$ is a subgroup of $Aut(G)$ and
$Aut_c^\otimes(G)\cap Inn(G)\cong Z_2^\otimes(G)/Z(G)$.
\end{cor}

\begin{proof}
It is clear that
$Aut_c^\otimes(G)\cap Inn(G)$ is a subgroup of $Aut(G)$.\\
Consider the map $\theta: Z_2^\otimes(G) \rightarrow Aut_c^\otimes(G) \cap Inn(G)$, defined by $\theta (g)=T_g.$
One can easily check that $\theta$ is an homomorphism with $\ker \theta = Z(G)$. Also, in view of Theorem 3.5, $\theta$
is epi and hence the result follows.
\end{proof}

\begin{thm} Let $G$ be a group such that $x\otimes x=1_\otimes$, for all $x \in G$. Then $(R_2^\otimes(G))'\subseteq
Z_2^\otimes(G)$.
\end{thm}

\begin{proof}
 Let $g,h \in R_2^\otimes(G).$ Then by Theorem
2.5, we have $T_g, T_h \in A^\otimes(G)\cap Inn(G)$. Thus applying Corollary 2.6, we have
$T_{[g,h]}, T_g \circ T_h \in A^\otimes(G)\cap Inn(G)$. For an arbitrary element $z$ in $G$, there exists element $y$ in $G$ such that $z=T_g \circ T_h(y)$. Then applying Lemma 2.2, for all $x \in G$, we have
\begin{eqnarray*}
T_{[g,h]}(x) \otimes z &=& (T_{[g,h]}(z) \otimes x)^{-1} \\
 &=& (T_h\circ T_g(y)\otimes x)^{-1}\\
 &=& T_h\circ T_g(x) \otimes y \\
 &=& (T_h(y) \otimes T_g(x))^{-1}\\
 &=& x \otimes T_g \circ T_h(y) \\
 &=& x \otimes z.
\end{eqnarray*}
Then Lemma 3.3 (i) implies that $T_{[g,h]} \in Aut_c^\otimes(G)$ and hence $[g,h] \in Z_2^\otimes (G)$ by Theorem 3.5.

 \end{proof}

\

% ------------------------------------------------------------------------
%\subsection*{Acknowledgment}

% ------------------------------------------------------------------------
\end{document}